\documentclass[reqno]{amsart}
\usepackage{amsfonts}
\usepackage{amssymb}
\usepackage{amsmath}
\usepackage{mathrsfs}
\usepackage{graphicx}
\usepackage{amsmath,amsthm,amsfonts,amssymb,latexsym,epsfig}
\usepackage{hyperref}

\newtheorem{theorem}{Theorem}[section]

\newtheorem{lemma}{Lemma}[section]

\newtheorem{prob}{Problem}[section]

\newtheorem{rem}{Remark}[section]

\numberwithin{equation}{section}

\theoremstyle{remark}

\begin{document}

\title{On the circumradius of a special class of $n$-simplices}
\author[Y.-D. Wu]{Yu-Dong Wu}
\address[Y.-D. Wu]{Department of Mathematics, Zhejiang Xinchang High School, Shaoxing, Zhejiang 312500, People's Republic of China}\email{yudong.wu@hotmail.com}
\author[Zh.-H. Zhang]{Zhi-Hua Zhang}
\address[Zh.-H. Zhang]{Department of Mathematics, Shili Senior High School in Zixing, Chenzhou, Hunan 423400, People's Republic of China}\email{zxzh1234@163.com}
\dedicatory{}

\begin{abstract}
An $n$-simplex is called circumscriptible (or edge-incentric) if
there is a sphere tangent to all its $n(n+1)/2$ edges. We obtain a
closed formula for the radius of the circumscribed sphere of the
circumscriptible $n$-simplex, and also prove a double inequality
involving the circumradius and the edge-inradius of such simplices.
Among this inequality settles affirmatively a part of a problem
posed by the authors.
\end{abstract}


\maketitle

%%--------------------------------------------------------------------

%%--------------------------------------------------------------------

\section{Introduction}

%%--------------------------------------------------------------------
A (non-degenerate) $n$-simplex $\Omega = [A_0,A_1, . . .,A_{n}], n
\geq 1$, is defined as the convex hull of $n + 1$ affinely
independent points (or position vectors) $A_0,A_1, . . .,A_n$ in
Euclidean $n$-space. The points $A_0,A_1, . . .,A_n$ are the
vertices of $\Omega$, and the line segments $a_{ij}$ joining two
different vertices $A_i$ and $A_j$ are its edges.
\par
Every $n$-simplex has a circumscribed sphere passing through its
$n+1$ vertices and an inscribed sphere tangent to each of its $n+1$
facets. For the circumradius $R$ and the inradius $r$, we have the
celebrated Euler's inequality as follows
\begin{equation}\label{eq001}
 R\geq nr.
\end{equation}

An $n$-simplex is circumscriptible (or edge-incentric) if there is a
sphere tangent to all its $n(n+1)/2$ edges. Considering such a
simplex, we call this the edge-tangent sphere of the $n$-simplex,
and note $\rho$ as the edge-inradius of this sphere. Of course, not
every $n$-simplex $(n\geq 3)$ has an edge-tangent sphere. However,
we have the following sufficient and necessary condition given by
Lin and Zhu \cite{lzc01} (see also Hajja \cite[p.~242, Theorem
4.1]{hajja01}).
\begin{theorem}\label{t01}
The $n$-simplex $\Omega$ has an edge-tangent sphere if and only if
there exist {\rm(}i.e., so-called the balloon radii{\rm)} $x_{i}>0$
with $0\leq i\leq n$ satisfying $a_{ij}=x_{i}+x_{j}$ for $0\leq i<
j\leq n$ or
$$x_i=\frac{1}{n(n-1)}\left(n\sum_{\substack{i=0\\(i\neq j)}}^{n}a_{ij}-\sum_{0\leq i<j\leq n}a_{ij}\right).$$
\end{theorem}

In 2006, Hajja \cite{hajja01} derived many geometrical properties of
the circumscriptible $n$-simplex. He also proved a closed formula
involving the edge-inradius (i.e., the radius of the edge-tangent
sphere):
\begin{theorem}\label{t02}{\rm(\cite[p.~249, Theorm 7.2 (d)]{hajja01})}
The edge-inradius $\rho$ of $\Omega$ is given by
\begin{equation}\label{02}
{\rho}^2=\frac{2(n-1)}{\left(\sum_{i=0}^{n}\limits
{\frac{1}{x_{i}}}\right)^2-(n-1)\sum_{i=0}^{n}\limits{\frac{1}{x_{i}^2}}}.
\end{equation}
\end{theorem}
The original for Theorem \ref{t02} is based on the following
generalized formula of the edge-inradius of a circumscriptible
$n$-simplex in terms of its edge-lengths given by Ivanoff
\cite{ivan01}, Lin and Zhu \cite{lzc01}.
\begin{theorem}\label{t03}
Given a circumscriptible $n$-simplex $\Omega$, we have
\begin{equation*}
{\rho}^2=-\frac{|A|}{2|A_1|},
\end{equation*}
where
$$A=\left(%
\begin{array}{cccc}
  -2x_{0}^2   & 2x_{0}x_{1} &\cdots  & 2x_{0}x_{n}\\
  2x_{0}x_{1} & -2x_{1}^2   &\cdots  & 2x_{1}x_{n}\\
  \cdots &\cdots &\cdots  &\cdots \\
  2x_{0}x_{n} & 2x_{1}x_{n} &\cdots  & -2x_{n}^2\\
\end{array}%
\right),$$ and
$$A_{1}=\left(%
\begin{array}{cccc}
  0 & 1 & \cdots & 1\\
  1 & \cdot  & \cdots & \cdot\\
  \vdots & \vdots & A  & \vdots\\
  1 & \cdot  & \cdots & \cdot\\
\end{array}%
\right).$$
\end{theorem}

For the circumradius $R$ of a circumscriptible $n$-simplex $\Omega$,
in 2007, an interesting problem stems naturally from the above
investigations and the following Theorem \ref{t04} by Hajja
{\rm\cite[p.~261]{hajja01}}: Finding a closed formula for the
circumradius of a circumscriptible $n$-simplex in terms of its
balloon radii $x_i$ with $0\leq i\leq n$ as similarly \eqref{02}.
\begin{theorem}{\rm(see \cite{ivan01} and also \cite{wx01})}\label{t04}
For a circumscriptible $n$-simplex $\Omega$, then we have
\begin{equation*}
{R}^2=-\frac{|D|}{2|D_1|},
\end{equation*}
where
\begin{equation*}D=\begin{pmatrix}
  0   & (x_{1}+x_{0})^2& (x_{2}+x_{0})^2 &\cdots  & (x_{n}+x_{0})^2\\
  (x_{0}+x_{1})^2 & 0 & (x_{2}+x_{1})^2  &\cdots  & (x_{n}+x_{1})^2\\
  (x_{0}+x_{2})^2 & (x_{1}+x_{2})^2& 0   & \cdots & (x_{n}+x_{2})^2\\
  \vdots &\vdots &\vdots  &  &\vdots\\
  (x_{0}+x_{n})^2 & (x_{1}+x_{n})^2 &(x_{2}+x_{n})^2& \cdots  & 0\\
\end{pmatrix},
\end{equation*} and
$$D_1=\left(%
\begin{array}{cccc}
  0 & 1 & \cdots & 1\\
  1 & \cdot  & \cdots & \cdot\\
  \vdots & \vdots & D  & \vdots\\
  1 & \cdot  & \cdots & \cdot\\
\end{array}%
\right).$$
\end{theorem}

A double inequality for the radius of the circumscriptible
tetrahedron $\Omega=A_{0}A_{1}A_{2}A_{3}A_{4}$ and sharpening
Euler's inequality \eqref{eq001} is proved in \cite{lz01} and
\cite{wz01}
\begin{equation}\label{01}
R\geq \sqrt{3}\rho\geq 3r.
\end{equation}

As a generalization of inequality \eqref{01}, Wu and Zhang
\cite{wz01} posed an analogous problem for the circumscriptible
$n$-simplex.
\begin{prob}\label{c02}
In a circumscriptible $n$-simplex $\Omega$, prove or disprove that
\begin{equation}\label{03}
R\geq \sqrt{\frac{2n}{n-1}}\rho \geq nr.
\end{equation}
\end{prob}
Recently, Wu et al. \cite{wzw01} proved the right hand of double
inequality \eqref{03}.

In this paper, we will give a closed formula for the circumradius of
a circumscriptible $n$-simplex in terms of its balloon radii $x_i$
with $0\leq i\leq n$ as similarly \eqref{02}, and settle the left
hand of double inequality \eqref{03} affirmatively.

%%--------------------------------------------------------------------

\section{Main Results}

%%--------------------------------------------------------------------
\begin{theorem}\label{t05}
The radius $R$ of the circumscribed sphere of a circumscriptible
$n$-simplex $\Omega$ is given by
\begin{equation}\label{341}
\left(\frac{R}{\rho}\right)^2=\frac{[MP-(n-1)(n-3)]^2-[M^2-(n-1)N][P^2-(n-1)Q]}{16(n-1)^2},
\end{equation}
and
\begin{equation}\label{34}
R^2=\frac{[MP-(n-1)(n-3)]^2-[M^2-(n-1)N][P^2-(n-1)Q]}{8(n-1)[P^2-(n-1)Q]},
\end{equation}
where
\begin{equation}\label{zzh001}
M=\sum_{i=0}^{n}{x_{i}},~~ N=\sum_{i=0}^{n}{{x_{i}}
^{2}},~~P=\sum_{i=0}^{n}{\frac{1}{x_{i}}},~~\mbox{and}~~
Q=\sum_{i=0}^{n}{\frac{1}{x_{i}^{2}}}.
\end{equation}
\end{theorem}

\begin{rem}
When $n=2$, then, in the triangle, we have the well known formula
\begin{align*}
R&=\frac{(x_0+x_1)(x_1+x_2)(x_2+x_0)}{4\sqrt{x_0x_1x_2(x_0+x_1+x_2)}}\\
&=\frac{a_{01}a_{12}a_{02}}{\sqrt{(a_{01}+a_{12}+a_{02})
(a_{01}+a_{12}-a_{02})(a_{12}+a_{02}-a_{01})(a_{01}+a_{02}-a_{12})}}.
\end{align*}
\end{rem}

\begin{rem}In 2007, Hajja {\rm\cite[p.~261]{hajja01}} said:``the questions regarding non-regular edge-incentric
d-simplices in which the circumcenter and the incenter coincide were
not considered. We expect these questions to be rather difficult,
since we were unable to find a closed formula for the circumradius
of an edge-incentric d-simplex in terms of its balloon radii. Such a
formula in the form of a quotient of two determinants is given in
{\rm\cite{ivan01}}".

For the given formula \eqref{34}, in our private communication,
Hajja also said:``I am really very impressed that you succeeded in
finding a closed formula for $R$ --- I have tried to do so last year
but never was able to. Actually, I asked a colleague in Germany for
help but he failed too".
\end{rem}
\begin{theorem}\label{t06}
For circumscriptible $n$-simplex $\Omega$, we have
\begin{equation}\label{04}
0\leq R^2-\frac{2n}{n-1}\rho^2\leq (n+1)^2|OG|^2,
\end{equation}
where $O$ and $G$ are the circumcenter and the centroid of the
circumscriptible $n$-simplex $\Omega$, respectively.
\end{theorem}
\begin{rem}
The left hand of double inequality \eqref{04} is just the left hand
of double inequality \eqref{03}.
\end{rem}

%%--------------------------------------------------------------------

\section{Preliminary Results}

%%--------------------------------------------------------------------
Throughout this section, let $A,~A_1,~D,~D_1$ and $x_{i}$ for $0\leq
i\leq n$ be defined by the above section, $V$ be the volume of
$\Omega$, and
\begin{equation*}
B_{1}=
\begin{pmatrix}
x_{0}^2 &1\\
x_{1}^2 &1\\
x_{2}^2 &1\\
\vdots &\vdots\\
x_{n}^2 &1\\
\end{pmatrix},
B_{2}=
\begin{pmatrix}
1 & 1 & 1 &\cdots & 1 \\
x_{0}^2 & x_{1}^2 & x_{2}^2 &\cdots & x_{n}^2\\
\end{pmatrix}.
\end{equation*}

\begin{lemma}{\rm(\cite[Lemma 1]{wzw01})}\label{l02}We have
\begin{equation}\label{27}
|A|=(-1)^n(n-1)2^{2n+1}\left(\prod_{i=0}^{n}x_{i}\right)^2,
\end{equation}
and
\begin{equation}\label{281}
A^{-1}=
\begin{pmatrix}
 \frac{2-n}{4n-4}\cdot\frac{1}{x_{0}^2}   & \frac{1}{4n-4}\cdot\frac{1}{x_{1}x_{0}} & \frac{1}{4n-4}\cdot\frac{1}{x_{2}x_{0}} &\cdots  & \frac{1}{4n-4}\cdot\frac{1}{x_{n}x_{0}}\\
  \frac{1}{4n-4}\cdot\frac{1}{x_{0}x_{1}} & \frac{2-n}{4n-4}\cdot\frac{1}{x_{1}^2}   & \frac{1}{4n-4}\cdot\frac{1}{x_{2}x_{1}} &\cdots  & \frac{1}{4n-4}\cdot\frac{1}{x_{n}x_{1}}\\
  \frac{1}{4n-4}\cdot\frac{1}{x_{0}x_{2}} & \frac{1}{4n-4}\cdot\frac{1}{x_{1}x_{2}} & \frac{2-n}{4n-4}\cdot\frac{1}{x_{2}^2}   &\cdots  & \frac{1}{4n-4}\cdot\frac{1}{x_{n}x_{2}}\\
  \vdots &\vdots &\vdots  &  &\vdots\\
  \frac{1}{4n-4}\cdot\frac{1}{x_{0}x_{n}} & \frac{1}{4n-4}\cdot\frac{1}{x_{1}x_{n}} & \frac{1}{4n-4}\cdot\frac{1}{x_{2}x_{n}} &\cdots  & \frac{2-n}{4n-4}\cdot\frac{1}{x_{n}^2}\\
\end{pmatrix}.
\end{equation}
\end{lemma}
\begin{proof}It is clear that
\begin{align*}
|A|&=2^{n+1}\left(\prod_{i=0}^{n}x_{i}\right)^2
\begin{vmatrix}
 -1 & 1  & 1  &\cdots  & 1\\
  1 & -1 & 1  &\cdots  & 1\\
  1 & 1  & -1 &\cdots  & 1\\
 \vdots &\vdots &\vdots  &  &\vdots\\
  1 & 1  & 1  &\cdots  & -1\\
\end{vmatrix}\end{align*}\begin{align*}
\\&=2^{n+1}\left(\prod_{i=0}^{n}x_{i}\right)^2
\begin{vmatrix}
 -1 & 1  & 1  &\cdots  & 1\\
  2 & -2 & 0  &\cdots  & 0\\
  2 & 0  & -2 &\cdots  & 0\\
 \vdots &\vdots &\vdots  &  &\vdots\\
  2 & 0  & 0  &\cdots  & -2\\
\end{vmatrix}
\\
&=2^{n+1}\left(\prod_{i=0}^{n}x_{i}\right)^2
\begin{vmatrix}
 n-1 & 1  & 1  &\cdots  & 1\\
  0 & -2 & 0  &\cdots  & 0\\
  0 & 0  & -2 &\cdots  & 0\\
 \vdots &\vdots &\vdots  &  &\vdots\\
  0 & 0  & 0  &\cdots  & -2\\
\end{vmatrix}
\\&=2^{n+1}\left(\prod_{i=0}^{n}x_{i}\right)^2\cdot(-1)^n(n-1)2^n
\\&=(-1)^n(n-1)2^{2n+1}\left(\prod_{i=0}^{n}x_{i}\right)^2.
\end{align*}

Let us to compute $A^{-1}$. We know that the adjoint of $A$ is
\begin{equation*}
A^{*}=
\begin{pmatrix}
A_{11} & A_{12} & \cdots & A_{1,n+1}\\
A_{21} & A_{22} & \cdots & A_{2,n+1}\\
\vdots & \vdots & \vdots  & \vdots   \\
A_{n+1,1} & A_{n+1,2} & \cdots & A_{n+1,n+1}\\
\end{pmatrix},
\end{equation*}
where $A_{ij}$ is the cofactor of the element $2x_{i-1}x_{j-1}(i\neq
j)$ or $-2x_{i-1}^2(i=j)$ of $A$.
\par
From the process of computing $|A|$ above, for $0\leq j\leq n$, it
is easily to obtain that
\begin{equation*}
A_{jj}=(-1)^{n-1}(n-2)2^{2n-1}\left(\prod_{\begin{subarray}{c}{i=0,}\\i\neq
j-1\end{subarray}}^{n}x_{i}\right)^2=\frac{2-n}{4(n-1)}\cdot\frac{1}{x_j^2}\cdot|A|.
\end{equation*}
Now, we compute $A_{ij}$ for $0\leq i<j\leq n$ because of
$A_{ij}=A_{ji}$ with $A=A'$. That is
\begin{align*}
A_{ij}&=(-1)^{i+j}
\begin{vmatrix}
  -2x_{0}^2   & 2x_{1}x_{0}& \cdots & 2x_{j-2}x_{0} &2x_{j}x_{0} &\cdots  & 2x_{n}x_{0}\\
  2x_{0}x_{1} & -2x_{1}^2  & \cdots & 2x_{j-2}x_{1} &2x_{j}x_{1} &\cdots  & 2x_{n}x_{1}\\
  \vdots &\vdots&   & \vdots & \vdots & &\vdots\\
  2x_{0}x_{i-2} & 2x_{1}x_{i-2}& \cdots &2x_{j-2}x_{i-2} &2x_{j}x_{i-2}   &\cdots  & 2x_{n}x_{i-2}\\
  2x_{0}x_{i} & 2x_{1}x_{i}& \cdots & 2x_{j-2}x_{i}   &2x_{j}x_{i}  &\cdots  & 2x_{n}x_{i}\\
  \vdots &\vdots &   & \vdots & \vdots & &\vdots\\
  2x_{0}x_{n} & 2x_{1}x_{n}& \cdots & 2x_{j-2}x_{n} & 2x_{j}x_{n} &\cdots  & -2x_{n}^2\\
\end{vmatrix}\\
&=(-1)^{i+j}2^{n}\frac{\left(\prod_{i=0}^{n}\limits
x_{i}\right)^2}{x_{i-1}x_{j-1}}
\begin{vmatrix}
-1 & 1 & 1& 1& \cdots&1\\
1 & -1 & 1& 1& \cdots&1\\
\vdots & \vdots & \ddots & & &\vdots\\
1& 1&  & C_{1} & &1\\
\vdots&\vdots&  & & \ddots &\vdots\\
1& 1& 1& 1& \cdots & -1\\
\end{vmatrix}
\\&
=(-1)^{i+j}2^{n}\frac{\left(\prod_{i=0}^{n}\limits
x_{i}\right)^2}{x_{i-1}x_{j-1}}
\begin{vmatrix}
-2 & & & & &\\
& -2 & & & &\\
& & \ddots & & &\\
& &  & C_{2} & &\\
& &  & & \ddots &\\
& &  & & & -2\\
\end{vmatrix}
\\
&=(-1)^{i+j}2^{n}(-2)^{n-(j-i)}|C_{2}|\frac{\left(\prod_{i=0}^{n}\limits
x_{i}\right)^2}{x_{i-1}x_{j-1}}\\&
=(-1)^{i+j}2^{n}(-2)^{n-(j-i)}(-1)^{j-i+1}(-2)^{j-i-1}\frac{\left(\prod_{i=0}^{n}\limits x_{i}\right)^2}{x_{i-1}x_{j-1}}\\
\end{align*}
\begin{align*}
&=(-1)^{2j+1}2^{n}(-2)^{n-1}\frac{\left(\prod_{i=0}^{n}\limits
x_{i}\right)^2}{x_{i-1}x_{j-1}}\\&
=(-1)^{n}2^{2n-1}\frac{\left(\prod_{i=0}^{n}\limits
x_{i}\right)^2}{x_{i-1}x_{j-1}}
=\frac{1}{4(n-1)}\cdot\frac{1}{x_{i-1}x_{j-1}}\cdot|A|,
\end{align*}
where
\begin{equation*}
C_{1}=
\begin{pmatrix}
1&-2&1&\cdots&1\\
1&1&-2&\cdots&1\\
\vdots&\vdots&\vdots&\ddots&\vdots\\
1&1&1&\cdots&-2\\
1&1&1&\cdots&1\\
\end{pmatrix}_{(j-i)\times(j-i)},
\end{equation*}
and
\begin{equation*}
C_{2}=
\begin{pmatrix}
0&-2&0&\cdots&0\\
0&0&-2&\cdots&0\\
\vdots&\vdots&\vdots&\ddots&\vdots\\
0&0&0&\cdots&-2\\
1&0&0&\cdots&0\\
\end{pmatrix}_{(j-i)\times(j-i)}.
\end{equation*}
It follows that \eqref{281} is true from the above $A^*,A_{ij}$ and
\begin{equation*}\label{28}
A^{-1}=\frac{1}{|A|}A^{*}.
\end{equation*}
\end{proof}
\begin{rem}
Wu et al. {\rm\cite{wzw01}} directly gave $A^{-1}$ without computing
process, and we here give the complete proof of this lemma. And the
proof of {\rm\eqref{27}} is more simple than Hajja
{\rm\cite[p.~250]{hajja01}}.
\end{rem}

\begin{lemma}\label{l03}For $x_i>0$ with $0\leq i\leq n$,
\begin{align*}\begin{split}
|D|=(-1)^n&\frac{2^{2n-3}}{n-1}\left(\prod_{i=0}^{n}x_{i}\right)^2\\&\cdot\{[MP-(n-1)(n-3)]^2-[M^2-(n-1)N][P^2-(n-1)Q]\},
\end{split}\end{align*}
where$M, N, P,$ and $Q$ are given by \eqref{zzh001}.
\end{lemma}
We shall give two proofs of Lemma \ref{l03} as follows.
\begin{proof}[{\bf Proof {\rm 1}}]Denoting the $j-$th column by $P_j$ and the $j-$th row
by $Q_j$, we perform the following operations on $|D|$:
\begin{enumerate}
\item
We accession a new row $Q_1=\begin{pmatrix} 1& x_0^2   & x_1^2&
x_2^2 &\cdots  & x_n^2
\end{pmatrix}.$
\item
We subtract $Q_1$ from $Q_{j+1}$ for $j = 1,\cdots, n+1$.
\item
We accession a new column $P_1=\begin{pmatrix} 1&0& x_0^2   & x_1^2&
x_2^2 &\cdots  & x_n^2
\end{pmatrix}^{'}.$
\item
We subtract $P_1$ from $P_{j+2}$ for $j = 1,\cdots, n+1$.
\item
 We divide $Q_{j+2}$ and $P_{j+2}$ by taking appropriate common factor $x_{j-1}$ for $j = 1,\cdots, n+1$.
\item
 We add $P_{j+3}$ from $P_3$ for $j = 1,\cdots, n$ and divide $P_3$ by taking a common factor $n-1$.
\item
We subtract $P_3$ from $P_{j+3}$ for $j = 1,\cdots, n$.
 \item
 We add $\frac{1}{4}\left(\frac{1}{n-1}P-\frac{1}{x_j}\right)Q_{j+3}$ from
 $Q_1$ and
 $\frac{1}{4}\left({x_j}-\frac{1}{n-1}M\right)Q_{j+3}$ from $Q_2$
 for $j =1, 2,\cdots, n$, and also add
 $\frac{1}{2}\left(\frac{1}{n-1}P
 -\frac{1}{x_0}\right)Q_{3}$ from
 $Q_1$ and
 $\frac{1}{2}\left({x_0}-\frac{1}{n-1}M\right)Q_{3}$ from
 $Q_2$.
\end{enumerate}
It is clear to see that
\begin{align*}
|D|&=
\begin{vmatrix}
  0   & (x_{1}+x_{0})^2& (x_{2}+x_{0})^2 &\cdots  & (x_{n}+x_{0})^2\\
  (x_{0}+x_{1})^2 & 0 & (x_{2}+x_{1})^2  &\cdots  & (x_{n}+x_{1})^2\\
  (x_{0}+x_{2})^2 & (x_{1}+x_{2})^2& 0   & \cdots & (x_{n}+x_{2})^2\\
  \vdots &\vdots &\vdots  &  &\vdots\\
  (x_{0}+x_{n})^2 & (x_{1}+x_{n})^2 &(x_{2}+x_{n})^2& \cdots  & 0\\
\end{vmatrix}\\
&=
\begin{vmatrix}1& x_0^2   & x_1^2& x_2^2 &\cdots  & x_n^2\\
 0& 0   & (x_{1}+x_{0})^2& (x_{2}+x_{0})^2 &\cdots  & (x_{n}+x_{0})^2\\
  0& (x_{0}+x_{1})^2 & 0 & (x_{2}+x_{1})^2  &\cdots  & (x_{n}+x_{1})^2\\
  0& (x_{0}+x_{2})^2 & (x_{1}+x_{2})^2& 0   & \cdots & (x_{n}+x_{2})^2\\
  \vdots & \vdots &\vdots &\vdots  &  &\vdots\\
  0& (x_{0}+x_{n})^2 & (x_{1}+x_{n})^2 &(x_{2}+x_{n})^2& \cdots  & 0\\
\end{vmatrix}\\
&=
\begin{vmatrix}1& x_0^2   & x_1^2& x_2^2 &\cdots  & x_n^2\\
 -1& -x_0^2   & x_{0}^2+2x_{0}x_{1}& x_{0}^2+2x_{0}x_{2} &\cdots  & x_{0}^2+2x_{0}x_{n}\\
  -1& x_{1}^2+2x_{0}x_{1} & -x_1^2  &  x_{1}^2+2x_{1}x_{2} &\cdots  & x_{1}^2+2x_{1}x_{n}\\
  -1& x_{2}^2+2x_{0}x_{2}& x_{2}^2+2x_{1}x_{2}& -x_2^2   & \cdots & x_{2}^2+2x_{2}x_{n}\\
  \vdots & \vdots &\vdots &\vdots  &  &\vdots\\
   -1& x_{n}^2+2x_{0}x_{n}& x_{n}^2+2x_{1}x_{n}&x_{n}^2+2x_{1}x_{n}  & \cdots  &  -x_n^2\\
\end{vmatrix}\\
&=
\begin{vmatrix}1&0& 0   & 0& 0 &\cdots  & 0\\
  0&1& x_0^2   & x_1^2& x_2^2 &\cdots  & x_n^2\\
  x_0^2&-1& -x_0^2   & x_{0}^2+2x_{0}x_{1}& x_{0}^2+2x_{0}x_{2} &\cdots  & x_{0}^2+2x_{0}x_{n}\\
  x_1^2&-1& x_{1}^2+2x_{0}x_{1} & -x_1^2  &  x_{1}^2+2x_{1}x_{2} &\cdots  & x_{1}^2+2x_{1}x_{n}\\
  \vdots &\vdots & \vdots &\vdots &\vdots  &  &\vdots\\
  x_n^2&-1& x_{n}^2+2x_{0}x_{n}& x_{n}^2+2x_{1}x_{n}&x_{n}^2+2x_{1}x_{n}  & \cdots  &  -x_n^2\\
\end{vmatrix}
\\&=
\begin{vmatrix}1&0& -1   & -1& -1 &\cdots  & -1\\
  0&1& x_0^2   & x_1^2& x_2^2 &\cdots  & x_n^2\\
  x_0^2&-1& -2x_0^2   & 2x_{0}x_{1}& 2x_{0}x_{2} &\cdots  & 2x_{0}x_{n}\\
  x_1^2&-1& 2x_{0}x_{1} & -2x_1^2  &  2x_{1}x_{2} &\cdots  & 2x_{1}x_{n}\\
  \vdots &\vdots & \vdots &\vdots &\vdots  &  &\vdots\\
  x_n^2&-1& 2x_{0}x_{n}& 2x_{1}x_{n}&2x_{1}x_{n}  & \cdots  &  -2x_n^2\\
\end{vmatrix}
\\
&=\prod_{i=0}^n x_i^2
\begin{vmatrix}1&0& -\frac{1}{x_0}   & -\frac{1}{x_1}& -\frac{1}{x_2}&\cdots  & -\frac{1}{x_n}\\
  0&1& x_0   & x_1& x_2 &\cdots  & x_n\\
  x_0&-\frac{1}{x_0}& -2   & 2& 2 &\cdots  & 2\\
  x_1&-\frac{1}{x_1}& 2 & -2  &  2 &\cdots  & 2\\
  \vdots &\vdots & \vdots &\vdots &\vdots  &  &\vdots\\
  x_n&-\frac{1}{x_n}& 2& 2&2  & \cdots  &  -2\\
\end{vmatrix}\\
\end{align*}
\begin{align*}
&=(n-1)\prod_{i=0}^nx_i^2
\begin{vmatrix}1&0& -\frac{1}{n-1}P   & -\frac{1}{x_1}& -\frac{1}{x_2}
&\cdots  & -\frac{1}{x_n}\\
  0&1&\frac{1}{n-1}M   & x_1& x_2 &\cdots  & x_n\\
  x_0&-\frac{1}{x_0}& 2   & 2& 2 &\cdots  & 2\\
  x_1&-\frac{1}{x_1}& 2 & -2  &  2 &\cdots  & 2\\
  x_2&-\frac{1}{x_2}& 2& 2& -2  & \cdots & 2\\
  \vdots &\vdots & \vdots &\vdots &\vdots  &  &\vdots\\
  x_n&-\frac{1}{x_n}& 2& 2&2  & \cdots  &  -2\\
\end{vmatrix}\\
&=(n-1)\prod_{i=0}^nx_i^2
\begin{vmatrix}1&0& -\frac{1}{n-1}P   &
\frac{1}{n-1}P-\frac{1}{x_1}&\cdots
& \frac{1}{n-1}P-\frac{1}{x_n}\\
  0&1&\frac{1}{n-1}M   & x_1-\frac{1}{n-1}M& \cdots
  & x_n-\frac{1}{n-1}M\\
  x_0&-\frac{1}{x_0}& 2   & 0 &\cdots  & 0\\
  x_1&-\frac{1}{x_1}& 2 & -4   &\cdots  & 0\\
  \vdots &\vdots & \vdots &\vdots   &  &\vdots\\
  x_n&-\frac{1}{x_n}& 2& 0 & \cdots  &  -4\\
\end{vmatrix}\\
&=(n-1)\prod_{i=0}^nx_i^2
\begin{vmatrix}{\frac{1}{4(n-1)}}X_1
& {\frac{1}{4(n-1)}}X_2&0 & 0& 0&\cdots &0\\
{\frac{1}{4(n-1)}}X_3& {\frac{1}{4(n-1)}}X_1&0& 0& 0&\cdots
  & 0\\  x_0&-\frac{1}{x_0}& 2   & 0& 0 &\cdots  & 0\\
  x_1&-\frac{1}{x_1}& 2 & -4  &  0 &\cdots  & 0\\
  x_2&-\frac{1}{x_2}& 2& 0& -4  & \cdots & 0\\
  \vdots &\vdots & \vdots &\vdots &\vdots  &  &\vdots\\
  x_n&-\frac{1}{x_n}& 2& 0&0  & \cdots  &  -4\\
\end{vmatrix}\\&
=(-1)^n\frac{2^{2n-3}}{n-1}\left(\prod_{i=0}^{n}x_{i}\right)^2\cdot(X_1^2-X_2X_3).
\end{align*}
where
\begin{align}\begin{split}\label{zzhe3320}
X_1&=MP-(n-1)(n-3),\\
X_2&=P^2-
(n-1)Q,\\
X_3&=M^2-(n-1)N, \end{split}\end{align} and $M, N, P, Q$ are given
by \eqref{zzh001}.
\end{proof}
\begin{proof}[{\bf Proof {\rm 2}}]
It is easily to find that
\begin{align}\label{20}\begin{split}
\begin{vmatrix}
A & B_{1}\\
-B_{2} & E_{2}\\
\end{vmatrix}
&=\begin{vmatrix}
A & B_{1}\\
0 & E_{2}+B_{2}A^{-1}B_{1}\\
\end{vmatrix}
=|A|\cdot|E_{2}+B_{2}A^{-1}B_{1}|,\end{split}
\end{align}
and
\begin{align}\label{21}\begin{split}
\begin{vmatrix}
A & B_{1}\\
-B_{2} & E_{2}\\
\end{vmatrix}
&=\begin{vmatrix}
A+B_{1}B_{2} & B_{1}\\
0 & E_{2}\\
\end{vmatrix}
=|A+B_{1}B_{2}|,\end{split}
\end{align}
where $E_{2}=\begin{pmatrix}1 &0 \\ 0 &1
\end{pmatrix}$.

By means of \eqref{281}, we obtain
\begin{align}\label{zzh33331}\begin{split}
&E_{2}+B_{2}A^{-1}B_{1}=
\begin{pmatrix}
\frac{1}{4(n-1)}X_1 & \frac{1}{4(n-1)}X_2\\
\frac{1}{4(n-1)}X_3& \frac{1}{4(n-1)}X_1
\end{pmatrix}.\end{split}
\end{align}

From \eqref{27} and \eqref{20}--\eqref{zzh33331}, it is deduced that
\begin{align*}
|D|=&|A+B_{1}B_{2}|=|A|\cdot|E_{2}+B_{2}A^{-1}B_{1}|
\\=&(-1)^n\frac{2^{2n-3}}{n-1}\left(\prod_{i=0}^{n}x_{i}\right)^2(X_1^2-X_2X_3).
\end{align*}
where $X_1,X_2$ and $X_3$ are given by \eqref{zzhe3320}.

This evidently completes the proof of Lemma \ref{l03}.
\end{proof}

\begin{lemma}{\rm(\cite[Corollary 2,~p.~96]{wx01})}\label{l04}
For $n$-simplex $\Omega$, we have
\begin{equation*}
(n!)^2V^2R^2=-\det\left(-\frac{1}{2}a_{ij}^2\right)=\frac{(-1)^n|D|}{2^{n+1}}.
\end{equation*}
\end{lemma}

\begin{lemma}{\rm(\cite[Corollary 1]{lzc01})}\label{l05}
Given a circumscriptible $n$-simplex $\Omega$, we have
\begin{equation*}\label{11}
(n!)^2V^2\rho^2=2^n(n-1)\left(\prod_{i=0}^{n}x_{i}\right)^2.
\end{equation*}
\end{lemma}

\begin{lemma}{\rm(\cite[~(3.5.11),~p.~112]{wx01})}\label{l06}
Let $O$ and $G$ are the circumcenter and the centroid of the
$n$-simplex $\Omega$, respectively. Then we have
\begin{equation*}
|OG|^2=R^2-\frac{1}{(n+1)^2}\sum_{0\leq i<j\leq n}a_{ij}^2.
\end{equation*}
\end{lemma}

%%--------------------------------------------------------------------
\section{The Proof of Theorem \ref{t05}}
%%--------------------------------------------------------------------
\begin{proof}
This follow straightforwardly from Theorem \textbf{II} and Lemmas
\ref{l03}--\ref{l05} by standard arguments.
\end{proof}

%%--------------------------------------------------------------------
\section{The Proof of Theorem \ref{t06}}
%%--------------------------------------------------------------------
\begin{proof}We will prove Theorem \ref{t06} with two steps. \begin{enumerate}
\item [(i)] Firstly, we prove the left hand of inequality
\eqref{04}. \par Let $M, N, P,$ and $Q$ are given by \eqref{zzh001}.
By using the well-known power mean inequality and Cauchy inequality,
then we have $N\geq \frac{M^2}{n+1},Q\geq \frac{P^2}{n+1},$ and
$MP\geq (n+1)^2$. Further considering $P^2-(n-1)Q>0$ follows that
\begin{align}\label{zzh30}\begin{split}
[M^2&-(n-1)N][P^2-(n-1)Q]\\&\leq
\left[M^2-\frac{n-1}{n+1}M^2\right]\left[P^2-\frac{n-1}{n+1}P^2\right]=\frac{4}{(n+1)^2}M^2P^2.
\end{split}\end{align} From \eqref{zzh30}, $1-\frac{4}{(n+1)^2}>0$
and the function
$$y=\left[1-\frac{4}{(n+1)^2}\right]\left[x-\frac{(n-3)(n+1)^2}{n+3}\right]^2-\frac{4(n-1)(n-3)^2}{n+3}$$
is increasing on interval
$\left(\frac{(n-3)(n+1)^2}{n+3},+\infty\right)$, and $MP\geq
(n+1)^2\geq \frac{(n-3)(n+1)^2}{n+3}$ for $n\geq 2$, we obtain
\begin{align}\label{zzh31}
&[MP-(n-1)(n-3)]^2-[M^2-(n-1)N][P^2-(n-1)Q]\nonumber\\  \geq
&[MP-(n-1)(n-3)]^2-\frac{4}{(n+1)^2}M^2P^2\nonumber\\
=&\left[1-\frac{4}{(n+1)^2}\right]\left[MP-\frac{(n-3)(n+1)^2}{n+3}\right]^2
-\frac{4(n-1)(n-3)^2}{n+3} \\ \geq
&\left[1-\frac{4}{(n+1)^2}\right]\left[(n+1)^2-\frac{(n-3)(n+1)^2}{n+3}\right]^2
-\frac{4(n-1)(n-3)^2}{n+3}\nonumber\\ =&32n(n-1).\nonumber
\end{align}

According to \eqref{341} and \eqref{zzh31}, we get
\begin{equation*}
\left(\frac{R}{\rho}\right)^2\geq\frac{32n(n-1)}{16(n-1)^2}=\frac{2n}{n-1}.
\end{equation*}
It is clear to show that the left of inequality \eqref{04} holds.

\item [(ii)] Secondly, with Lemma \ref{l06}, the right hand of inequality \eqref{04}
is
\begin{equation}\label{05}
(n+1)^2R^2-\sum_{0\leq i<j\leq n}{a_{ij}^2}\geq
R^2-\frac{2n}{n-1}\rho^2.
\end{equation}
For $n=3$, the required result is proved in \cite{lz01}.

Now we prove that inequality \eqref{05} holds when $n\geq 4$.
Obviously, from Theorems \ref{t02}--\ref{t05}, inequality \eqref{05}
is equivalent to
\begin{align*}
n(n+2)R^2&+\frac{2n}{n-1}\rho^2 \geq \sum_{0\leq i<j\leq
n}{a_{ij}^2},\end{align*} or
\begin{align*}\frac{n(n+2)}{8(n-1)}&\cdot\frac{[MP-(n-1)(n-3)]^2-[M^2-(n-1)N][P^2-(n-1)Q]}{P^2-(n-1)Q}\nonumber\\
&+\frac{4n}{P^2-(n-1)Q}\geq M^2+(n-1)N,
\end{align*}
that is
\begin{align}\label{52}\begin{split}
&n(n+2)[MP-(n-1)(n-3)]^2+32n(n-1)\\
\geq &[(n^2+10n-8)M^2-(n-1)(n-2)(n-4)N][P^2-(n-1)Q],\end{split}
\end{align}
where $M, N, P,$ and $Q$ are given by \eqref{zzh001}.

From the proof of Theorem \ref{t05}, we have $N\geq \frac{M^2}{n+1}$
and $Q\geq \frac{P^2}{n+1}$. Hence, in order to prove inequality
\eqref{52}, we only need to prove the following inequality
\begin{align*}
n(n+2)[MP-(n-1)(n-3)]^2+32n(n-1)\geq \frac{12n(3n-2)}{(n+1)^2}M^2P^2
\end{align*}
or
\begin{align}\begin{split}\label{54}
(n^2+5n-26)&\left[MP-\frac{(n+2)(n-3)(n+1)^2}{n^2+5n-26}\right]^2\\&-\frac{4(3n-10)^2(n+1)^4}{n^2+5n-26}\geq
0.\end{split}
\end{align}
When $n\geq 4(n\in N)$, it's clear that
$(n+1)^2>\frac{(n+2)(n-3)(n+1)^2}{n^2+5n-26}$, and the function
$$f(x)=(n^2+5n-26)\left[x-\frac{(n+2)(n-3)(n+1)^2}{n^2+5n-26}\right]^2-\frac{4(3n-10)^2(n+1)^4}{n^2+5n-26}$$
is increasing on interval $[(n+1)^2, +\infty)$. Thus, from $MP\geq
(n+1)^2$, we get
\begin{align*}
f(MP)\geq f((n+1)^2)=0.
\end{align*}
It is just as inequality \eqref{54}. Further, inequality \eqref{52}
or \eqref{05} holds. \par The proof of Theorem \ref{t06} is thus
completed.
\end{enumerate}
\end{proof}

\section{Acknowledgements}
The authors would like to thank Professor M. Hajja for his kindly
help in sending his valuable paper to them.

\end{document}